\documentclass[10pt,reqno]{amsart}

\usepackage{amsthm, mathrsfs, amsmath, amstext, amsxtra, amsfonts, dsfont, amssymb}
\usepackage{lmodern}
\usepackage[dvipsnames]{xcolor}
\usepackage[colorlinks, linkcolor=red, citecolor=blue, urlcolor=OliveGreen, pagebackref, hypertexnames=false]{hyperref}

\usepackage[belowskip=2pt,aboveskip=0pt]{caption}
\usepackage{booktabs}
\newcommand{\N}{\mathbb N} 
\newcommand{\Z}{\mathbb Z} 
\newcommand{\R}{\mathbb R} 
\newcommand{\C}{\mathbb C} 
 
\newcommand{\vareps}{\varepsilon}
\newcommand{\sct} {s_{\text{c}}}
\newcommand{\sce} {s_{\emph{c}}}
\newcommand{\gct} {\gamma_{\text{c}}}
\newcommand{\im}[1]{\mbox{Im} \ #1} 
\newcommand{\scal}[1]{\left\langle #1 \right\rangle} 

\newcommand{\defendproof}{\hfill $\Box$} 

\newtheorem{theorem}{Theorem}[section]
\newtheorem{lem}[theorem]{Lemma} 
\newtheorem{prop}[theorem]{Proposition}
\newtheorem{coro}[theorem]{Corollary} 
\theoremstyle{definition}

\newtheorem{rem}[theorem]{Remark}

\title[Blow-up criteria NLFS]{Blow-up criteria for fractional nonlinear Schr\"odinger equations} 

\author[V. D. Dinh]{Van Duong Dinh}
\address[V. D. Dinh]{Institut de Math\'ematiques de Toulouse UMR5219, Universit\'e Toulouse CNRS, 31062 Toulouse Cedex 9, France and Department of Mathematics, HCMC University of Pedagogy, 280 An Duong Vuong, Ho Chi Minh, Vietnam}
\email{dinhvan.duong@math.univ-toulouse.fr}


\begin{document}

\begin{abstract}
We consider the focusing fractional nonlinear Schr\"odinger equation
\[
i\partial_t u - (-\Delta)^s u = -|u|^\alpha u, \quad (t,x) \in \R^+ \times \R^d,
\]
where $s \in (1/2,1)$ and $\alpha>0$. By using localized virial estimates, we establish general blow-up criteria for non-radial solutions to the equation. As consequences, we obtain blow-up criteria in both $L^2$-critical and $L^2$-supercritical cases which extend the results of Boulenger-Himmelsbach-Lenzmann [{\it Blowup for fractional NLS}, J. Funct. Anal. 271 (2016), 2569--2603] for non-radial initial data. 
\end{abstract}

\maketitle

\section{Introduction}
\setcounter{equation}{0}
We consider the Cauchy problem for fractional nonlinear Schr\"odinger equations with focusing power-type nonlinearity
\begin{align}
\left\{
\begin{array}{rcl}
i\partial_t u - (-\Delta)^s u &=& - |u|^{\alpha} u, \quad (t,x) \in \R^+ \times \R^d, \\
u(0) &=& u_0, 
\end{array}
\right.
\label{focusing FNLS}
\end{align}
where $u$ is a complex valued function defined on $\R^+ \times \R^d$, $d\geq 1$, $s \in (1/2,1)$ and $\alpha>0$. The operator $(-\Delta)^s$ is the fractional Laplacian which is defined by $\mathcal{F}^{-1}[|\xi|^{2s} \mathcal{F}]$ with $\mathcal{F}$ and $\mathcal{F}^{-1}$ the Fourier transform and its inverse respectively. The equation $(\ref{focusing FNLS})$ can be seen as a canonical model for a nonlocal dispersive PDE with focusing nonlinearity that can exhibit standing waves and wave collapse. The fractional Schr\"odinger equation was first discovered by Laskin \cite{Laskin} as a result of extending the Feynmann path integral, from the Brownian-like to L\'evy-like quantum mechanical paths. The fractional Schr\"odinger equation also appears in the continuum limit of discrete models with long-range interactions (see e.g. \cite{KirkLenzStaf}) and in the description of Boson stars as well as in water wave dynamics (see e.g. \cite{FrohJonsLenz} or \cite{IonePusa}). 

In the last decade, the fractional nonlinear Schr\"odinger equation $(\ref{focusing FNLS})$ has attrated a lot of interest in mathematics, numerics and physics (see e.g. \cite{BouHimLen, CaiMajMcLauTab, Dinh-fract, Dinh-blow-dyn-mfnls, Dinh-blow-dyn-ifnls, GuoSireWangZhao, GuoZhu,  HongSire, IonePusa, KleiSparMark, PengShi, SunWangYaoZheng} and references therein). The local well-posedness in Sobolev spaces for non-radial data was established by Hong-Sire in \cite{HongSire} (see also \cite{Dinh-fract}). The local well-posedness for radial $H^s$ data was studied by the author in \cite{Dinh-blow-dyn-mfnls}. The existence of radial finite time blow-up $H^s$ solutions was established recently by Boulenger-Himmelsbach-Lenzmann in \cite{BouHimLen}. Dynamics of finite time blow-up solutions were studied in \cite{Dinh-blow-dyn-mfnls, Dinh-blow-dyn-ifnls}. The sharp threshold of blow-up and scattering in the $L^2$-supercritical and $\dot{H}^s$-subcritical case was first considered by Sun-Wang-Yao-Zheng in \cite{SunWangYaoZheng}. This result was then extended by Guo-Zhu in \cite{GuoZhu}. The orbital stability as well as the orbital instability was proved by Peng-Shi in \cite{PengShi}. Recently, the author in \cite{Dinh-instability} proved the strong instability of standing waves for the equation in the $L^2$-supercritical case.

Before stating our main results, let us recall some basic facts of $(\ref{focusing FNLS})$. The equation $(\ref{focusing FNLS})$ has formally the conservation of mass and the energy:
\begin{align*}
M(u(t)) &= \int |u(t,x)|^2 dx = M(u_0), & \text{(Mass)}\\
E(u(t)) &= \frac{1}{2} \int |(-\Delta)^{s/2} u(t,x)|^2 dx - \frac{1}{\alpha+2} \int |u(t,x)|^{\alpha+2}dx=E(u_0). & \text{(Energy)}
\end{align*}
The equation $(\ref{focusing FNLS})$ also enjoys the scaling invariance
\[
u_\lambda(t,x) = \lambda^{\frac{2s}{\alpha}} u(\lambda^{2s} t, \lambda x), \quad \lambda>0. 
\]
A calculation shows
\[
\|u_\lambda(0)\|_{\dot{H}^{\nu}} = \lambda^{\nu+\frac{2s}{\alpha} - \frac{d}{2}} \|u_0\|_{\dot{H}^\nu}.
\]
We thus define the critical exponent
\begin{align}
\sct := \frac{d}{2} -\frac{2s}{\alpha}. \label{critical exponent}
\end{align}

The local well-posedness for $(\ref{focusing FNLS})$ in Sobolev spaces for non-radial data was studied by Hong-Sire in \cite{HongSire} (see also \cite{Dinh-fract}). Note that the unitary group $e^{-it(-\Delta)^s}$ enjoys several types of Strichartz estimates: non-radial Strichartz estimates (see e.g. \cite{ChoOzawaXia} or \cite{Dinh-fract}); radial Strichartz estimates (see e.g. \cite{GuoWang}, \cite{Ke} or \cite{ChoLee}); and weighted Strichartz estimates (see e.g. \cite{FangWang}). For non-radial data, these Strichartz estimates have a loss of derivatives. This makes the study of local well-posedness more difficult and leads to a weak local theory comparing to the standard nonlinear Schr\"odinger equation (e.g. $s=1$ in $(\ref{focusing FNLS})$). One can remove the loss of derivatives in Strichartz estimates by considering radially symmetric initial data. However, these Strichartz estimates without loss of derivatives require an restriction on the validity of $s$, that is $s \in \left[\frac{d}{2d-1},1\right)$. Since we are interested in blow-up criteria for solutions of $(\ref{focusing FNLS})$ in general Sobolev spaces $H^\gamma$, we first need to establish the local well-posedness in such spaces. This will lead to a regularity condition on the nonlinearity, that is, 
\begin{align}
\lceil \gamma \rceil \leq \alpha + 1, \label{regularity assumption}
\end{align}
where $\lceil \gamma \rceil$ is the smallest positive integer greater than or equal to $\gamma$. We refer the reader to Section $\ref{section local theory}$ for more details. 

Recently, Boulenger-Himmelsbach-Lenzmann in \cite{BouHimLen} established blow-up criteria for radial $H^s$ solutions to $(\ref{focusing FNLS})$. More precisely, they proved the following:
\begin{theorem}[Radial blow-up criteria \cite{BouHimLen}] \label{theorem BouHimLen}
	Let $d\geq 2$, $s \in (1/2, 1)$ and $\alpha>0$. Let $u_0 \in H^s$ be radial and assume that the corresponding solution to $(\ref{focusing FNLS})$ exists on the maximal time interval $[0,T)$. 
	\begin{itemize}
		\item {\bf Mass-critical case}, i.e. $\alpha = \frac{4s}{d}$: If $E(u_0)<0$, then the solution $u$ either blows up in finite time, i.e. $T<+\infty$ or blows up infinite time, i.e. $T=+\infty$ and 
		\[
		\|(-\Delta)^{s/2}u(t)\|_{L^2} \geq c t^s, \quad \forall t \geq t_0,
		\]
		with some $C>0$ and $t_0>0$ that depend only on $u_0, s$ and $d$. 
		\item {\bf Mass and energy intercritical case}, i.e. $\frac{4s}{d}<\alpha<\frac{4s}{d-2s}$: If $\alpha<4s$ and either $E(u_0)<0$, or if $E(u_0) \geq 0$, we assume that
		\[
		\left\{
		\begin{array}{ccc}
		E^{\sce}(u_0) M^{s-\sce}(u_0) &<& E^{\sce}(Q) M^{s-\sce}(Q), \\ \|(-\Delta)^{s/2}u_0\|_{L^2}^{\sce} \|u_0\|_{L^2}^{s-\sce} &>& \|(-\Delta)^{s/2} Q\|^{\sce}_{L^2} \|Q\|_{L^2}^{s-\sce},
		\end{array}
		\right.
		\]
		where $Q$ is the unique (up to symmetries) positive radial solution to the elliptic equation
		\begin{align}
		(-\Delta)^s Q + Q - |Q|^\alpha Q=0, \label{elliptic equation intercritical}
		\end{align}
		then the solution blows up in finite time, i.e. $T<+\infty$. 
		\item {\bf Energy-critical case}, i.e. $\alpha=\frac{4s}{d-2s}$: If $\alpha<4s$ and either $E(u_0)<0$, or if $E(u_0)\geq 0$, we assume that
		\[
		\left\{
		\begin{array}{ccc}
		E(u_0) &<& E(W), \\ 
		\|(-\Delta)^{s/2}u_0\|_{L^2} &>& \|(-\Delta)^{s/2} W\|_{L^2},
		\end{array}
		\right.
		\]
		where $W$ is the unique (up to symmetries) positive radial solution to the elliptic equation
		\begin{align}
		(-\Delta)^s W - |W|^{\frac{4s}{d-2s}} W=0, \label{elliptic equation energy-critical}
		\end{align}
		then the solution blows up in finite time, i.e. $T<+\infty$. 
	\end{itemize}
\end{theorem}
Note that the uniqueness (up to symmetries) of positive radial solution to $(\ref{elliptic equation intercritical})$ and $(\ref{elliptic equation energy-critical})$ were proved in \cite{FrankLenzmann,FrankLenzmannSilvestre}.

The main purposes of this paper is to show blow-up criteria for non-radial $H^\gamma$ solutions for $(\ref{focusing FNLS})$. Before entering some details of our results, let us recall known blow-up criteria for the focusing nonlinear Schr\"odinger equation
\begin{align}
i \partial_t u + \Delta u = -|u|^\alpha u, \quad u(0) = u_0. \label{focusing NLS}
\end{align}
The existence of finite time blow-up $H^1$ solutions for $(\ref{focusing NLS})$ was first proved by Glassey in \cite{Glassey}. More precisely, he proved that for any negative $H^1$ initial data satisfying $x u_0 \in L^2$, the corresponding solution blows up in finite time. Ogawa-Tsutsumi in \cite{OgawaTsutsumi, OgawaTsutsumi1d} showed the existence of blow-up solutions for negative radial data in dimensions $d\geq 2$ and for negative data (not necessary radially symmetry) in the one dimensional case. Holmer-Roudenko in \cite{HolmerRoudenko} showed that in the mass and energy intercritical case, if initial data satisfies $E(u_0) \geq 0$ and
\begin{align}
\left\{
\begin{array}{ccc}
E^{\gct}(u_0) M^{1-\gct}(u_0) &<& E^{\gct}(R) M^{1-\gct}(R), \\ 
\|\nabla u_0\|_{L^2}^{\gct} \|u_0\|_{L^2}^{1-\gct} &>& \|\nabla R\|^{\gct}_{L^2} \|R\|_{L^2}^{1-\gct},
\end{array}
\right. \label{blowup condition NLS}
\end{align}
and in addition if $xu_0 \in L^2$ or $u_0$ is radial with $N\geq 2$ and $\alpha<4$, then the corresponding solution blows up in finite time. Here $R$ is the ground state of $(\ref{focusing NLS})$ which is the unique (up to symmetries) positive radial solution of the elliptic equation
\[
\Delta R - R + |R|^\alpha R=0,
\]
and $\gct= \frac{d}{2}-\frac{2}{\alpha} \in (0,1)$. Later, Holmer-Roudenko in \cite{HolmerRoudenko-diver} showed that if $H^1$ initial data (not necessary finite-variance or radially symmetry) satisfies $(\ref{blowup condition NLS})$, then the corresponding solution either blows up in finite time or it blows up infinite time in the sense that there exists a sequence of times $t_n \rightarrow +\infty$ such that $\|\nabla u(t_n)\|_{L^2} \rightarrow \infty$. Recently, Du-Wu-Zhang extended the result of \cite{HolmerRoudenko-diver} and proved a blow-up criterion for $(\ref{focusing NLS})$ with initial data (without finite-variance and radially symmetric assumptions) in the energy-critical and energy-supercritical cases. 
	
Inspiring by the idea of Du-Wu-Zhang, we study the blow-up criteria for the focusing fractional nonlinear equation $(\ref{focusing FNLS})$. The main difficulty is the appearance of the fractional order Laplacian $(-\Delta)^s$. When $s=1$, one can compute easily the time derivative of the virial action, which is
\begin{align}
\frac{d}{dt} \left(\int \varphi |u(t)|^2 dx\right) = 2 \im{\int \nabla \varphi \cdot \nabla u(t) \overline{u}(t) dx}. \label{virial action}
\end{align}
Using this identity, Du-Wu-Zhang \cite{DuWuZhang} derive an $L^2$-estimate in the exterior ball. Thanks to this $L^2$-estimate and the virial estimates, they prove the result. In the case $s\in (1/2,1)$, the identity $(\ref{virial action})$ does not hold. However, by exploiting the idea of \cite{BouHimLen} with the use of the Balakrishman's formula, namely
\[
(-\Delta)^s = \frac{\sin \pi s}{\pi} \int_0^\infty m^{s-1} \frac{-\Delta}{-\Delta +m} dm, \quad s \in (0,1),
\]
we are able to compute the time derivative of the virial action (see Lemma $\ref{lemma localized virial identity I}$):
\begin{align*}
\frac{d}{dt} \left(\int \varphi |u(t)|^2 dx\right) = &-i\int_0^\infty m^s \int \Delta \varphi |u_m(t)|^2 dx dm \\
&-2i \int_0^\infty m^s \int \overline{u}_m(t) \nabla \varphi \cdot \nabla u_m(t) dx dm,
\end{align*}
where $u_m(t)$ is an auxiliary function defined by
\[
u_m(t) = \sqrt{\frac{\sin \pi s}{\pi}} \frac{1}{-\Delta+m} u(t).
\]
This identity plays a similar role as in $(\ref{virial action})$, and we can show the blow-up criteria for $(\ref{focusing FNLS})$ with non-radial initial data.

Denote
\begin{align}
K(u(t)):=\frac{s}{2} \|(-\Delta)^{s/2} u(t)\|^2_{L^2} - \frac{d\alpha}{4(\alpha+2)} \|u(t)\|^{\alpha+2}_{L^{\alpha+2}}. \label{define blowup quantity}
\end{align}
We will see in $(\ref{virial identity})$ that the quantity $K(u(t))$ is related to the following virial identity
\[
\frac{d}{dt} \left( 4 \im{\int \overline{u}(t) x \cdot \nabla u(t) dx} \right) = 16K(u(t)).
\]
\begin{theorem}[Blow-up criteria] \label{theorem blowup criteria}
	Let $d\geq 1$, $s\in (1/2,1)$ and $\gamma \geq \max\{s,\sce\}$. Let $\alpha>0$ be such that if $\alpha$ is not an even integer, then $(\ref{regularity assumption})$ holds. Let $u_0 \in H^\gamma$ be such that the corresponding (not necessary radial) solution $u$ to $(\ref{focusing FNLS})$ exists on the maximal time $[0,T)$. If there exists $\delta>0$ such that 
	\begin{align}
	\sup_{t\in [0,T)} K(u(t)) \leq -\delta <0, \label{blowup condition}
	\end{align}
	then one of the following statements holds true:
	\begin{itemize}
	\item $u(t)$ blows up in finite time in the sense $T<+\infty$ must hold;
	\item $u(t)$ blows up infinite time and 
	\begin{align}
	\sup_{t\in [0,+\infty)} \|u(t)\|_{L^q} =\infty, \label{divergence Lq norm}
	\end{align}
	for any $q>\alpha+2$. In particular, there exists a time sequence $(t_n)_{n}$ such that $t_n \rightarrow +\infty$ and 
	\[
	\lim_{n\rightarrow \infty} \|u(t_n)\|_{L^q}=\infty,
	\]
	for any $q>\alpha+2$.
	\end{itemize}
\end{theorem}
\begin{rem} \label{remark blowup criteria}
	\begin{itemize}
		\item The condition $\gamma \geq s$ ensures the solution enjoys the conservation of mass and energy. 
		\item It is still possible (see e.g. \cite[Remark 6.5.9]{Cazenave}) that there exists a solution which blows up in finite positive time is global in negative time and vice versa.
		\item In the case $T<+\infty$, we learn from the local theory  that if $\gamma>\sct$, then $\lim_{t\uparrow T} \|u(t)\|_{H^\gamma} =\infty$.
		
	\end{itemize}
\end{rem}
The following result gives blow-up criteria for solutions with negative energy initial data. 
\begin{coro} \label{corollary negative energy blowup criteria}
	Let $d\geq 1$, $s\in (1/2,1)$ and $\gamma\geq \max\{s,\sce\}$. Let $\alpha \geq \frac{4s}{d}$ be such that if $\alpha$ is not an even integer, then $(\ref{regularity assumption})$ holds. Let $u_0 \in H^\gamma$ be such that the corresponding (not necessary radial) solution to $(\ref{focusing FNLS})$ exists on the maximal time $[0,T)$. If $E(u_0)<0$, then one of the following statements holds true:
	\begin{itemize}
		\item $u(t)$ blows up in finite time in the sense $T<+\infty$ must hold;
		\item $u(t)$ blows up infinite time and 
		\[
		\sup_{t\in [0,+\infty)} \|u(t)\|_{L^q} =\infty,
		\]
		for any $q >\alpha+2$. In particular, there exists a time sequence $(t_n)_n$ such that $t_n \rightarrow +\infty$ and 
		\[
		\lim_{n\rightarrow \infty} \|u(t_n)\|_{L^q} =\infty,
		\]
		for any $q>\alpha+2$.
	\end{itemize}
\end{coro}
This corollary follows directly from Theorem $\ref{theorem blowup criteria}$ with the fact 
\[
K(u(t)) = sE(u(t)) - \frac{d\alpha-4s}{4(\alpha+2)} \|u(t)\|^{\alpha+2}_{L^{\alpha+2}} \leq sE(u_0) <0.
\]
Let us now consider $\gamma=s$. Note that in this case the regularity condition $(\ref{regularity assumption})$ is no longer needed. We firstly have the following blow-up criteria in the mass-critical case.
\begin{prop} [Mass-critical blow-up criteria] \label{proposition mass-critical blowup criteria}
	Let $d\geq 1$ and $s\in (1/2,1)$. Let $u_0 \in H^s$ be such that the corresponding (not necessary radial) solution to the mass-critical $(\ref{focusing FNLS})$, i.e. $\alpha=\frac{4s}{d}$ exists on the maximal time $[0,T)$. If $E(u_0)<0$, then one of the following statements holds true:
	\begin{itemize}
		\item $u(t)$ blows up in finite time, i.e. $T<+\infty$ and 
		\[
		\lim_{t\uparrow T} \|(-\Delta)^{s/2} u(t)\|_{L^2} =\infty;
		\]
		\item $u(t)$ blows up infinite time and 
		\[
		\sup_{t\in [0,+\infty)} \|u(t)\|_{L^q} =\infty,
		\]
		for any $q\geq \frac{4s}{d}+2$. In particular, there exists a time sequence $(t_n)_n$ such that $t_n \rightarrow +\infty$ and
		\[
		\lim_{n\rightarrow \infty} \|(-\Delta)^{s/2} u(t_n)\|_{L^2} =\infty.
		\]
	\end{itemize}
\end{prop}

Now, let $\alpha_\star<\alpha<\alpha^\star$, where
\[
\alpha_\star:=\frac{4s}{d}, \quad \alpha^\star:= \left\{ 
\begin{array}{cl}
\infty &\text{if } d=1, \\
\frac{4s}{d-2s} &\text{if } d\geq 2.
\end{array}
\right.
\]
In the mass and energy intercritical case, we have the following blow-up criteria.
\begin{prop}[Mass and energy intercritical blow-up criteria] \label{proposition intercritical blowup criteria}
	Let $d\geq 1$ and $s\in(1/2,1)$. Let $u_0 \in H^s$ be such that the corresponding (not necessary radial) solution to the mass and energy intercritical $(\ref{focusing FNLS})$, i.e. $\alpha_\star<\alpha<\alpha^\star$ exists on the maximal time $[0,T)$. If either
	\[
	E(u_0)<0,
	\]
	or if $E(u_0) \geq 0$, we assume that
	\begin{align}
	\left\{
	\begin{array}{ccc}
	E^{\sce}(u_0) M^{s-\sce}(u_0) &<& E^{\sce}(Q) M^{s-\sce}(Q), \\
	\|(-\Delta)^{s/2} u_0\|^{\sce}_{L^2} \|u_0\|^{s-\sce}_{L^2} &>& \|(-\Delta)^{s/2}Q\|^{\sce}_{L^2} \|Q\|^{s-\sce}_{L^2},
	\end{array}
	\right.
	\label{blowup condition intercritical}
	\end{align} 
	where $Q$ is the unique (up to symmetries) positive radial solution to $(\ref{elliptic equation intercritical})$, then one of the following statements holds true:
	\begin{itemize}
		\item $u(t)$ blows up in finite time, i.e. $T<+\infty$ and 
		\[
		\lim_{t\uparrow T} \|(-\Delta)^{s/2} u(t)\|_{L^2} =\infty;
		\]
		\item $u(t)$ blows up infinite time and 
		\[
		\sup_{t\in[0,+\infty)} \|u(t)\|_{L^q} = \infty,
		\]
		for any $q \geq \alpha+2$. In particular, there exists a time sequence $(t_n)_n$ such that $t_n \rightarrow +\infty$ and 
		\[
		\lim_{n\rightarrow \infty} \|(-\Delta)^{s/2} u(t_n)\|_{L^2} =\infty.
		\]
	\end{itemize}
\end{prop}
Finally, we have the following blow-up criteria in the energy-critical case.
\begin{prop}[Energy-critical blow-up criteria] \label{proposition energy-critical blowup criteria}
Let $d\geq 2$ and $s\in (1/2,1)$. Let $u_0 \in H^s$ be such that the corresponding (not necessary radial) solution to the energy-critical $(\ref{focusing FNLS})$, i.e. $\alpha=\frac{4s}{d-2s}$ exists on the maximal time $[0,T)$. If either
\[
E(u_0) <0,
\]
or if $E(u_0)\geq 0$, we assume that
\begin{align}
\left\{
\begin{array}{ccc}
E(u_0) &<& E(W), \\
\|(-\Delta)^{s/2} u_0\|_{L^2} &>& \|(-\Delta)^{s/2} W\|_{L^2},
\end{array}
\right.
\label{blowup condition energy-critical}
\end{align}
where $W$ is the unique (up to symmetries) positive radial solution to $(\ref{elliptic equation energy-critical})$, then one of the following statements holds true:
\begin{itemize}
	\item $u(t)$ blows up in finite time, i.e. $T<+\infty$;
	\item $u(t)$ blows up infinite time and
	\[
	\sup_{t\in [0, +\infty)} \|u(t)\|_{L^q} =\infty,
	\]
	for any $q>\frac{2d}{d-2s}$. In particular, there exists a time sequence $(t_n)_n$ such that $t_n \rightarrow +\infty$ and 
	\[
	\lim_{n \rightarrow \infty} \|u(t_n)\|_{L^q} =\infty,
	\]
	for any $q>\frac{2d}{d-2s}$.
\end{itemize}
\end{prop}

The paper is oganized as follows. In Section $\ref{section preliminaries}$, we recall some preliminaries related to the fractional nonlinear Schr\"odinger equation such as Strichartz estimates and nonlinear estimates. 
In Section $\ref{section local theory}$, we recall the local well-posedness for $(\ref{focusing FNLS})$ in general Sobolev spaces $H^\gamma$ with non-radial and radial initial data. In Section $\ref{section virial estimates}$, we prove various virial-type estimates related to the equation. The blow-up criteria for non-radial solutions of $(\ref{focusing FNLS})$ will be proved in Section $\ref{section blowup criteria}$.
\section{Preliminaries} \label{section preliminaries}
\setcounter{equation}{0}
\subsection{Strichartz estimates}
In this subsection, we recall Strichartz estimates for the fractional Schr\"odinger equation. Let $I \subset \R$ and $p,q \in [1, \infty]$. We define the Strichartz norm
\[
\|f\|_{L^p(I, L^q)} := \Big( \int_I \Big( \int_{\R^d} |f(t,x)|^q dx \Big)^{\frac{p}{q}}\Big)^{\frac{1}{p}},
\]
with a usual modification when either $p$ or $q$ are infinity. Let $\chi_0$ be a bump function supported in $\{\xi \in \R^d \ : \ |\xi| \leq 2\}$ and $\chi_0(\xi) =1$ for $|\xi| \leq 1$. Set $\chi(\xi) = \chi_0(\xi) - \chi_0(2\xi)$. We denote the Littlewood-Paley projections $P_0 := \chi_0(D), P_N := \chi(N^{-1}D)$ with $N= 2^k, k \in \Z$, where $\chi_0(D)f = \mathcal{F}^{-1}[\chi_0 \mathcal{F}(f)]$ and similarly for $\chi(N^{-1}D)$ with $\mathcal{F}$ and $\mathcal{F}'$ the Fourier transform and its inverse respectively. Given $\gamma \in \R$ and $1\leq q \leq \infty$, one defines the Besov space $B^\gamma_q$ as
\[
B^\gamma_q:= \left\{ u \in \mathscr{S}' \ : \ \|u\|_{B^\gamma_q} := \|P_0 u\|_{L^q} + \Big( \sum_{N \in 2^{\N}} N^{2\gamma} \|P_N u\|^2_{L^q}\Big)^{1/2} <\infty \right\},
\]
where $\mathscr{S}'$ is the space of tempered distributions. There are several types of Strichartz estimates for the Schr\"odinger operator $e^{-it(-\Delta)^s}$. We recall below two-types of Strichartz estimates for the fractional Schr\"odinger equation:

	\underline{For general data} (see e.g. \cite{ChoOzawaXia} or \cite{Dinh-fract}): the following estimates hold for $d\geq 1$ and $s \in (0,1) \backslash \{1/2\}$,
	\begin{align}
	\|e^{-it(-\Delta)^s} \psi\|_{L^p(\R, L^q)} &\lesssim \||\nabla|^{\gamma_{p,q}} \psi\|_{L^2}, \label{strichartz estimate homogeneous}\\
	\Big| \int_0^t e^{-i(t-\tau) (-\Delta)^s} f(\tau) d\tau\Big|_{L^p(\R, L^q)} &\lesssim \||\nabla|^{\gamma_{p,q} -\gamma_{a', b'} - 2s} f\|_{L^{a'}(\R, L^{b'})}, \label{strichartz estimate inhomogeneous}
	\end{align}
	where $(p,q)$ and $(a,b)$ are Schr\"odinger admissible, i.e.
	\[
	p \in [2,\infty], \quad q \in [2, \infty), \quad (p,q, d) \ne (2,\infty, 2), \quad \frac{2}{p} + \frac{d}{q} \leq \frac{d}{2},
	\]
	and 
	\[
	\gamma_{p,q} = \frac{d}{2} -\frac{d}{q} -\frac{2s}{p},
	\]
	similarly for $\gamma_{a', b'}$. Here $(a, a')$ and $(b,b')$ are conjugate pairs. It is worth noticing that for $s \in (0,1) \backslash \{1/2\}$ the admissible condition $\frac{2}{p} +\frac{d}{q} \leq \frac{d}{2}$ implies $\gamma_{p,q}>0$ for all admissible pairs $(p,q)$ except $(p,q)=(\infty, 2)$. This means that the above Strichartz estimates have a loss of derivatives. In the local theory, this loss of derivatives makes the problem more difficult, and leads to a weak local well-posedness result comparing to the nonlinear Schr\"odinger equation (see Section $\ref{section local theory}$). 
	
	\underline{For radially symmetric data} (see e.g. \cite{Ke}, \cite{GuoWang} or \cite{ChoLee}): the estimates $(\ref{strichartz estimate homogeneous})$ and $(\ref{strichartz estimate inhomogeneous})$ hold true for $d\geq 2, s \in (0,1) \backslash \{1/2\}$ and $(p,q), (a,b)$ satisfy the radial Sch\"odinger admissible condition:
	\begin{align*}
	p \in [2, \infty], \quad q \in [2, \infty), \quad (p,q) \ne \left(2, \frac{4d-2}{2d-3}\right), \quad \frac{2}{p} + \frac{2d-1}{q} \leq \frac{2d-1}{2}.
	\end{align*}
	Note that the admissible condition $\frac{2}{p} + \frac{2d-1}{q} \leq \frac{2d-1}{2}$ allows us to choose $(p,q)$ so that $\gamma_{p,q} =0$. More precisely, we have for $d\geq 2$ and $\frac{d}{2d-1}\leq s<1$ \footnote{This condition follows by pluging $\gamma_{p,q}=0$ to $\frac{2}{p}+\frac{2d-1}{q} \leq \frac{2d-1}{2}$.},
	\begin{align}
	\|e^{-it(-\Delta)^s} \psi\|_{L^p(\R, L^q)} &\lesssim \|\psi\|_{L^2}, \label{radial strichartz estimate homogeneous}\\
	\Big| \int_0^t e^{-i(t-\tau) (-\Delta)^s} f(\tau) d\tau\Big|_{L^p(\R, L^q)} &\lesssim \| f\|_{L^{a'}(\R, L^{b'})}, \label{radial strichartz estimate inhomogeneous}
	\end{align}
	where $\psi$ and $f$ are radially symmetric and $(p,q), (a,b)$ satisfy the fractional admissible condition, 
	\begin{align}
	p \in [2,\infty], \quad q \in [2, \infty), \quad (p,q) \ne \left(2, \frac{4d-2}{2d-3}\right), \quad \frac{2s}{p} + \frac{d}{q} = \frac{d}{2}. \label{fractional admissible}
	\end{align}
	These Strichartz estimates with no loss of derivatives allow us to give a similar local well-posedness result as for the nonlinear Schr\"odinger equation (see again Section $\ref{section local theory}$). 

\subsection{Nonlinear estimates} 
We recall the following fractional chain rule which is needed in the local well-posedness for $(\ref{focusing FNLS})$. 
\begin{lem}[Fractional chain rule \cite{ChristWeinstein}] \label{lem nonlinear estimate}
	Let $F \in C^1(\C, \C)$ and $\gamma \in (0,1)$. Then for $1<q \leq q_2 <\infty$ and $1<q_1 \leq \infty$ satisfying $\frac{1}{q}=\frac{1}{q_1}+\frac{1}{q_2}$,
	\[
	\||\nabla|^\gamma F(u)\|_{L^q} \lesssim \|F'(u)\|_{L^{q_1}} \||\nabla|^\gamma u\|_{L^{q_2}}.
	\]
\end{lem} 
\section{Local well-posedness} \label{section local theory}
In this section, we recall the local well-posedness for $(\ref{focusing FNLS})$ in Sobolev spaces. The proof is based on the contraction mapping argument using Strichartz estimates. Due to the loss of derivatives in Strichartz estimates, we thus consider separately two cases: non-radial initial data and radially symmetric initial data. 
\subsection{Non-radial initial data.}
We have the following local well-posedness for $(\ref{focusing FNLS})$ in Sobolev spaces due to \cite{HongSire} (see also \cite{Dinh-fract}). Let us start with the local well-posedness in the sub-critical case, i.e. $\gamma>\sct$.
\begin{prop}[Non-radial local theory I \cite{HongSire, Dinh-fract}] \label{proposition lwp non radial}
	Let $d\geq 1$, $s\in (1/2,1)$ and $\alpha>0$. Let $\gamma\geq 0$ be such that
	\[
	\gamma > \left\{
	\begin{array}{cl}
	1/2 -2s/\max(\alpha, 4) &\text{if } d=1,\\
	d/2 - 2s/\max(\alpha,2) &\text{if } d\geq 2,
	\end{array}
	\right.
	\]
	and also, if $\alpha$ is not an even integer, $(\ref{regularity assumption})$ holds. Then for any $u_0 \in H^\gamma$, there exist $T\in (0,+\infty]$ and a unique solution to $(\ref{focusing FNLS})$ satisfying
	\[
	u \in C([0,T), H^\gamma) \cap L^p_{\emph{loc}}([0,T), L^\infty),
	\]
	for some
	\[
	p > \left\{
	\begin{array}{cl}
	\max(\alpha, 4) &\text{if } d=1,\\
	\max(\alpha,2) &\text{if } d\geq 2.
	\end{array}
	\right.
	\]
	Moreover, the following properties hold:
	\begin{itemize}
		\item If $T<\infty$, then $\lim_{t\uparrow T} \|u(t)\|_{H^\gamma} =\infty$;
		\item There is conservation of mass, i.e. $M(u(t)) = M(u_0)$ for all $t\in [0,T)$;
		\item If $\gamma\geq s$, then the energy is conserved, i.e. $E(u(t)) = E(u_0)$ for all $t\in [0,T)$.
	\end{itemize}	
\end{prop}
We refer the reader to \cite{Dinh-fract} (see also \cite{HongSire}) for the proof of above result. The proof is based on Strichartz estimates and the contraction mapping argument. Note that in the non-radial case, there is a loss of derivatives in Strichartz estimates. Fortunately, this loss of derivatives can be compensated for by using the Sobolev embedding. However, there is still a gap between $\sct$ and $1/2 -2s/\max(\alpha, 4)$ when $d=1$, and $d/2 - 2s/\max(\alpha,2)$ when $d\geq 2$. We also have the local well-posedness in the critical case, i.e. $\gamma=\sct$.
\begin{prop} [Non-radial local theory II \cite{HongSire, Dinh-fract}] \label{proposition lwp non radial critical}
	Let $d\geq 1$, $s\in (1/2,1)$ and 
	\[
	\alpha> \left\{
	\begin{array}{cl}
	4 &\text{if } d=1, \\
	2 &\text{if } d\geq 2,
	\end{array}
	\right.
	\]
	be such that $\sce \geq 0$, and also, if $\alpha$ is not an even integer, $(\ref{regularity assumption})$ holds. Then for any $u_0 \in H^{\sce}$, there exist $T\in (0,+\infty]$ and a unique solution to $(\ref{focusing FNLS})$ satisfying
	\[
	u \in C([0,T), H^{\sce}) \cap L^p_{\emph{loc}}([0,T), B^{\sce-\gamma_{p,q}}_q),
		\]
		where
		\[
		\left\{
		\begin{array}{c l}
		p=4, q=\infty &\text{if } d=1, \\
		2<p<\alpha, q=2p/(p-2) &\text{if } d=2, \\
		p=2, q=2d/(d-2) &\text{if } d\geq 3,
		\end{array}
		\right.
		\]
		and
		\[
		\gamma_{p,q} = \frac{d}{2}-\frac{d}{q} - \frac{2s}{p}.
		\]
		Moreover, the following properties hold:
		\begin{itemize}
			\item There is conservation of mass, i.e. $M(u(t)) = M(u_0)$ for all $t\in [0,T)$;
			\item If $\sce \geq s$, then the energy is conserved, i.e. $E(u(t)) = E(u_0)$ for all $t\in [0,T)$.
		\end{itemize}
	\end{prop}
We refer the reader to \cite{HongSire} (see also \cite{Dinh-fract}) for the proof of this result. Unlike the sub-critical case, the Sobolev embedding does not help us to overcome the loss of derivatives. It needs a delicate estimate on $L^\alpha_t L^\infty_x$ to overcome this loss of derivatives. 

\subsection{Radial initial data.} 
In this subsection, we show the local well-posedness for $(\ref{focusing FNLS})$ with radial initial data in Sobolev spaces. The proof is again based on the contraction mapping argument via Strichartz estimates. Thanks to Strichartz estimates without loss of derivatives in the radial case, we have better local well-posedness comparing to the non-radial case. Let us start with the local well-posedness in the subcritical case.
\begin{prop}[Radial local theory I] \label{proposition lwp radial subcritical}
	Let $d\geq 2$ and $s \in \left[\frac{d}{2d-1},1\right)$. Let $\gamma \in \left[0,\frac{d}{2}\right)$ be such that $\gamma >\sce$, and also, if $\alpha$ is not an even integer, $(\ref{regularity assumption})$ holds. Let
	\[
	p = \frac{4s(\alpha+2)}{\alpha(d-2\gamma)}, \quad q= \frac{d(\alpha+2)}{d+\alpha\gamma}.
	\] 
	Then for any $u_0 \in H^\gamma$ be radial, there exists $T \in (0,+\infty]$ and a unique solution to $(\ref{focusing FNLS})$ satisfying
	\[
	u \in C([0,T), H^\gamma) \cap L^p_{\emph{loc}} ([0,T), W^{\gamma,q}).
	\]
	Moreover, the following properties hold: 
	\begin{itemize}
	\item If $T<+\infty$, then $\lim_{t\uparrow T} \|u(t)\|_{H^\gamma} = \infty$;
	\item There is conservation of mass, i.e. $M(u(t)) = M(u_0)$ for all $t\in [0,T)$;
	\item If $\gamma\geq s$, then the energy is conserved, i.e. $E(u(t)) = E(u_0)$ for all $t\in [0,T)$.
	\end{itemize}
\end{prop}
\begin{proof}
	It is easy to check that $(p,q)$ satisfies the fractional admissible condition $(\ref{fractional admissible})$. We choose $(m,n)$ so that
	\begin{align}
	\frac{1}{p'} = \frac{1}{p} + \frac{\alpha}{m}, \quad \frac{1}{q'}= \frac{1}{q} + \frac{\alpha}{n}. \label{choice mn}
	\end{align}
	We see that
	\begin{align}
	\frac{\alpha}{m}-\frac{\alpha}{p} = 1-\frac{\alpha (d-2\gamma)}{4s} =:\theta>0, \quad q \leq n = \frac{dq}{d-\gamma q}. \label{define theta}
	\end{align}
	The later fact gives the Sobolev embedding $\dot{W}^{\gamma,q} \hookrightarrow L^n$. Let us now consider
	\[
	X:= \left\{ C(I, H^\gamma) \cap L^p(I, W^{\gamma,q}) \ : \ \|u\|_{L^\infty(I, \dot{H}^\gamma)} + \|u\|_{L^p(I, \dot{W}^{\gamma,q})} \leq M \right\},
	\]
	equipped with the distance
	\[
	d(u,v):= \|u-v\|_{L^\infty(I, L^2)} +\|u-v\|_{L^p(I, L^q)},
	\]
	where $I=[0,\zeta]$ and $M, \zeta>0$ to be chosen later. By Duhamel's formula, it suffices to prove that the functional 
	\[
	\Phi(u)(t) := e^{-it(-\Delta)^s} u_0 -i \mu \int_0^t e^{-i(t-\tau)(-\Delta)^s} |u(\tau)|^\alpha u(\tau) d\tau
	\]
	is a contraction on $(X,d)$. By radial Strichartz estimates $(\ref{radial strichartz estimate homogeneous})$ and $(\ref{radial strichartz estimate inhomogeneous})$,
	\begin{align*}
	\|\Phi(u)\|_{L^\infty(I, \dot{H}^\gamma)} + \|\Phi(u)\|_{L^p(I, \dot{W}^{\gamma, q})} &\lesssim \|u_0\|_{\dot{H}^\gamma} + \||u|^\alpha u\|_{L^{p'}(I, \dot{W}^{\gamma,q'})}, \\
	\|\Phi(u) - \Phi(v)\|_{L^\infty(I, L^2)} + \|\Phi(u) - \Phi(v)\|_{L^p(I, L^q)} &\lesssim \||u|^\alpha u - |v|^\alpha v\|_{L^{p'} (I, L^{q'})}.
	\end{align*}
	The fractional chain rule given in Lemma $\ref{lem nonlinear estimate}$ and the H\"older inequality give
	\begin{align*}
	\||u|^\alpha u\|_{L^{p'}(I, \dot{W}^{\gamma,q'})} &\lesssim \|u\|^\alpha_{L^m(I, L^n)} \|u\|_{L^p(I, \dot{W}^{\gamma,q})}, \\
	&\lesssim |I|^\theta \|u\|^\alpha_{L^p(I, L^n)} \|u\|_{L^p(I, \dot{W}^{\gamma,q})} \\
	&\lesssim |I|^\theta \|u\|^{\alpha+1}_{L^p(I, \dot{W}^{\gamma,q})}.
	\end{align*}
	Similarly,
	\begin{align*}
	\||u|^\alpha- |v|^\alpha v\|_{L^{p'}(I, L^{q'})} &\lesssim \left(\|u\|^\alpha_{L^m(I,L^n)} + \|v\|^\alpha_{L^m(I,L^n)} \right) \|u-v\|_{L^p(I,L^q)} \\
	&\lesssim |I|^\theta \left(\|u\|^\alpha_{L^p(I,\dot{W}^{\gamma,q})} + \|v\|^\alpha_{L^p(I,\dot{W}^{\gamma,q})} \right) \|u-v\|_{L^p(I,L^q)}.
	\end{align*}
	This shows that for all $u, v\in X$, there exists $C>0$ independent of $T$ and $u_0 \in H^\gamma$ such that
	\begin{align*}
	\|\Phi(u)\|_{L^\infty(I, \dot{H}^\gamma)} + \|\Phi(u)\|_{L^p(I, \dot{W}^{\gamma,q})} &\leq C\|u_0\|_{\dot{H}^\gamma} + C \zeta^\theta M^{\alpha+1}, \\
	d(\Phi(u),\Phi(v)) &\leq C\zeta^\theta M^\alpha d(u,v).
	\end{align*}
	If we set $M=2C\|u_0\|_{\dot{H}^\gamma}$ and choose $\zeta>0$ so that
	\[
	C\zeta^\theta M^\alpha \leq \frac{1}{2},
	\]
	then $\Phi$ is a strict contraction on $(X,d)$. This proves the existence of solution $u \in C(I, H^\gamma) \cap L^p(I, W^{\gamma,q})$. The blow-up alternative follows easily since the existence time depends only on the $\dot{H}^\gamma$-norm of initial data. The conservation of mass and energy follow by a standard approximation procedure. The proof is complete.
\end{proof}
Finally, we have the local well-posedness with radial initial data in the critical case.
\begin{prop}[Radial local theory II] \label{proposition lwp radial critical}
Let $d\geq 2$ and $s \in \left[\frac{d}{2d-1},1\right)$. Let $\alpha>0$ be such that $\sce \geq 0$, and also, if $\alpha$ is not an even integer, $(\ref{regularity assumption})$ holds. Let
\[
p= \alpha+2, \quad q = \frac{2d(\alpha+2)}{d(\alpha+2) - 4s}. 
\] 
Then for any $u_0 \in H^{\sce}$ radial, there exist $T \in (0,+\infty]$ and a unique solution to $(\ref{focusing FNLS})$ satisfying
\[
u \in C([0,T), H^{\sce} \cap L^p_{\emph{loc}} ([0,T), W^{\sce,q}).
\]
Moreover, the following properties hold:
\begin{itemize}
	\item There is conservation of mass, i.e. $M(u(t)) = M(u_0)$ for all $t\in [0,T)$;
	\item If $\sce \geq s$, then the energy is conserved, i.e. $E(u(t)) = E(u_0)$ for all $t\in [0,T)$.
\end{itemize}
\end{prop}
\begin{proof}
	It is easy to check that $(p,q)$ satisfies the fractional admissible condition. We next choose $n$ so that
	\[
	\frac{1}{q'} = \frac{1}{q} + \frac{\alpha}{n} \text{ or } n = \frac{dq}{d-\sct q}.
	\]
	The last condition ensures the Sobolev embedding 
	\begin{align}
	\|u\|_{L^p(I, L^n)} \lesssim \|u\|_{L^p(I, \dot{W}^{\sct,q})}. \label{sobolev embedding}
	\end{align}
	Let us consider
	\[
	X:= \left\{ u \in L^p(I, W^{\sct,q}) \ : \ \|u\|_{L^p(I, \dot{W}^{\sct,q})} \leq M \right\},
	\]
	equipped with the distance
	\[
	d(u,v) = \|u-v\|_{L^p(I, L^q)},
	\]
	where $I=[0,\zeta]$ and $M, \zeta>0$ to be chosen later. We will show that the functional $\Phi$ is a contraction on $(X,d)$, where
	\[
	\Phi(u)(t) = e^{-it(-\Delta)^s} u_0 - i \mu \int_0^t e^{-i(t-\tau)(-\Delta)^s} |u(\tau)|^\alpha u(\tau) d\tau =: u_{\text{hom}}(t) + u_{\text{inh}}(t). 
	\]
	By radial Strichartz estimate $(\ref{radial strichartz estimate homogeneous})$, we have
	\[
	\|u_{\text{hom}}\|_{L^p(I, \dot{W}^{\sct, q})} \lesssim \|u_0\|_{\dot{H}^{\sct}}.
	\]
	This shows that $\|u_{\text{hom}}\|_{L^p(I, \dot{W}^{\sct,q})} \leq \epsilon$ for some $\epsilon$ small enough provided that $\zeta$ is small or $\|u_0\|_{\dot{H}^{\sct}}$ is small. Similarly, by $(\ref{radial strichartz estimate inhomogeneous})$, we have
	\[
	\|u_{\text{inh}}\|_{L^p(I, \dot{W}^{\sct,q})} \lesssim \||u|^\alpha u\|_{L^{p'}(I, \dot{W}^{\sct,q'})}.
	\]
	By the fractional chain rule, the H\"older inequality and $(\ref{sobolev embedding})$, we get
	\[
	\||u|^\alpha u\|_{L^{p'}(I, \dot{W}^{\sct,q'})} \lesssim \|u\|^\alpha_{L^p(I, L^n)} \|u\|_{L^p(I, \dot{W}^{\sct,q})} \lesssim \|u\|^{\alpha+1}_{L^p(I, \dot{W}^{\sct,q})}.
	\]
	Similarly, we have
	\[
	\||u|^\alpha u - |v|^\alpha v\|_{L^{p'}(I, L^{q'})} \lesssim \left( \|u\|^\alpha_{L^p(I, \dot{W}^{\sct,q})} + \|v\|^\alpha_{L^p(I,\dot{W}^{\sct,q})} \right) \|u-v\|_{L^p(I,L^q)}. 
	\]
	Thus, for all $u,v \in X$, there exists $C$ independent of $u_0 \in H^{\sct}$ such that
	\begin{align*}
	\|\Phi(u)\|_{L^p(I, \dot{W}^{\sct,q})} &\leq \epsilon + C M^{\alpha+1}, \\
	d(\Phi(u), \Phi(v)) &\leq C M^\alpha d(u,v).
	\end{align*}
	If we choose $\epsilon, M>0$ small so that
	\[
	CM^\alpha \leq \frac{1}{2}, \quad \epsilon + \frac{M}{2} \leq M,
	\]
	then $\Phi$ is a contraction on $(X,d)$. This shows the existence of solutions. The conservation of mass and energy are standard and we omit the details. The proof is complete.
\end{proof}
\section{Virial estimates} \label{section virial estimates}
\setcounter{equation}{0}

In this section, we recall and prove some virial estimates related to $(\ref{focusing FNLS})$ which are in the same spirit as in \cite[Section 2]{BouHimLen}. Let us start with the following estimates.
\begin{lem} [\cite{BouHimLen}] \label{lemma various estimate 1}
	Let $d\geq 1$ and $\varphi:\R^d \rightarrow \R$ be such that $\nabla \varphi \in W^{1,\infty}$. Then for any $u \in H^{1/2}$, it holds that
	\begin{align}
	\left| \int \overline{u} (x) \nabla \varphi(x) \cdot \nabla u(x) dx \right| \leq C \|\nabla \varphi\|_{W^{1,\infty}} \left( \||\nabla|^{1/2} u\|^2_{L^2} + \|u\|_{L^2} \||\nabla|^{1/2} u\|_{L^2} \right), \label{various estimate 1}
	\end{align}
	for some constant $C>0$ depending only on $d$.
	\end{lem}
\begin{lem} \label{lemma various estimate 2}
	Let $d\geq 1$, $s\in (1/2,1)$ and $\varphi: \R^d \rightarrow \R$ be such that $\nabla \varphi \in W^{1,\infty}$. Then for any $u \in L^2$, it holds that
	\begin{align}
	\left| \int_0^\infty m^s \int \Delta \varphi |u_m|^2 dx dm\right| \leq C \|\Delta \varphi\|^{2s-1}_{L^\infty} \|\nabla \varphi\|^{2-2s}_{L^\infty} \|u\|^2_{L^2}, \label{various estimate 2}
	\end{align}
	for some constant $C>0$ depending only on $s$ and $d$. Here 
	\begin{align}
	u_m(x) = c_s \frac{1}{-\Delta+m} u(x) = c_s \mathcal{F}^{-1} \left(\frac{\hat{u}(\xi)}{|\xi|^2+m} \right), \quad m>0, \label{define auxiliary function}
	\end{align}
	where 
	\[
	c_s:= \sqrt{\frac{\sin \pi s}{\pi}}. 
	\]
\end{lem}
\begin{proof}
	The proof is essentially given in \cite[Lemma A.2]{BouHimLen}. For the reader's convenience, we give some details. We split the $m$-integral into $\int_0^{\tau}\cdots$ and $\int_{\tau}^\infty \cdots$ with $\tau>0$ to be chosen later. By integration by parts and H\"older's inequality, we learn that
	\begin{align*}
	\left| \int_0^\tau m^s \int \Delta \varphi |u_m|^2 dx dm \right| &= \left| \int_0^\tau m^s \int \nabla \varphi \cdot \left( \nabla u_m \overline{u}_m + u_m \nabla \overline{u}_m\right)dxdm\right| \\
	&= \|\nabla \varphi\|_{L^\infty} \int_0^\tau m^s \|\nabla u_m\|_{L^2} \|u_m\|_{L^2} dm \\
	&\lesssim \|\nabla \varphi\|_{L^\infty} \|u\|^2_{L^2} \left(\int_0^\tau m^{s-3/2} dm \right) \\
	&\lesssim \tau^{s-1/2} \|\nabla \varphi\|_{L^\infty} \|u\|^2_{L^2}.
	\end{align*} 
	Here we use the fact 
	\[
	\|\nabla u_m \|_{L^2} \lesssim m^{-1/2} \|u\|_{L^2}, \quad \|u_m\|_{L^2} \lesssim m^{-1} \|u\|_{L^2},
	\]
	which follows directly from the definition of $u_m$. On the other hand, 
	we find that
	\begin{align*}
	\left| \int_{\tau}^\infty m^s \int \Delta \varphi |u_m|^2 dx dm \right| & \lesssim \|\Delta \varphi\|_{L^\infty} \left(\int_\tau^\infty m^s \|u_m\|^2_{L^2} dm \right) \\
	&\lesssim \|\Delta\varphi\|_{L^\infty} \|u\|^2_{L^2} \left(\int_\tau^\infty m^{s-2} dm \right) \\
	&\lesssim \tau^{s-1}\|\Delta \varphi\|_{L^\infty} \|u\|^2_{L^2}.
	\end{align*}
	Collecting the above estimates, we obtain
	\[
	\left| \int_0^\infty m^s \int \Delta \varphi |u_m|^2 dx dm\right| \lesssim \left(\tau^{s-1/2} \|\nabla \varphi\|_{L^\infty} + \tau^{s-1} \|\Delta\varphi \|_{L^\infty}  \right) \|u\|^2_{L^2},
	\]
	for arbitrary $\tau>0$. Minimizing the right hand side with respect to $\tau$, i.e. choosing $\tau = \frac{(1-s)^2\|\Delta \varphi\|^2_{L^\infty}}{(s-1/2)^2\|\nabla \varphi\|^2_{L^\infty}}$, we complete the proof. 
\end{proof}
	\begin{lem} \label{lemma various estimate 3}
	Let $d\geq 1$, $s\in (1/2,1)$ and $\varphi: \R^d \rightarrow \R$ be such that $\nabla \varphi \in W^{1,\infty}$. Then for any $u \in H^{1/2}$, it holds that 
	\begin{align}
	\left| \int_0^\infty m^s \int \overline{u}_m \nabla \varphi \cdot \nabla u_m dx dm\right| \leq C \|\nabla \varphi\|_{W^{1,\infty}} \|u\|^2_{H^{1/2}}, \label{various estimate 3}
	\end{align}
	for some constant $C>0$ depending only on $d$, where $u_m$ is given in $(\ref{define auxiliary function})$. 
	\end{lem}
	\begin{proof}
		As in the proof of Lemma $\ref{lemma various estimate 2}$, we split the $m$-integral into two parts $\int_0^\tau\cdots$ and $\int_\tau^\infty \cdots$ with $\tau>0$ to be chosen shortly. By H\"older's inequality, we estimate the first term as
		\begin{align*}
		\left|\int_0^\tau m^s \int \overline{u}_m \nabla \varphi \cdot \nabla u_m dx dm \right| & \lesssim \|\nabla \varphi\|_{L^\infty} \int_0^\tau m^s \|u_m\|_{L^2} \|\nabla u_m\|_{L^2} dm \\
		&\lesssim \|\nabla \varphi\|_{L^\infty} \|u\|^2_{L^2} \left( \int_0^\tau m^{s-3/2} dm\right) \\
		&\lesssim \tau^{s-1/2} \|\nabla \varphi\|_{L^\infty} \|u\|^2_{L^2}.
		\end{align*}
		For the second term, we use \footnote{ The smoothing operator $(-\Delta+m)^{-1}$ implies that $u_m \in H^{\rho+2}$ whenever $u \in H^\rho$. Hence the hypotheses of Lemma $\ref{lemma various estimate 1}$ are satisfied.} Lemma $\ref{lemma various estimate 1}$ to get
		\begin{align*}
		\left| \int_\tau^\infty m^s \right.& \left.\int \overline{u}_m \nabla \varphi \cdot \nabla u_m dx dm  \right| \\
		&\lesssim \|\nabla \varphi\|_{W^{1,\infty}} \int_\tau^\infty m^s \left(\||\nabla|^{1/2} u_m\|^2_{L^2} + \|u_m\|_{L^2} \||\nabla|^{1/2} u_m\|_{L^2} \right) dm \\
		&\lesssim \|\nabla \varphi\|_{W^{1,\infty}}  \left(\||\nabla|^{1/2} u\|^2_{L^2} + \|u\|_{L^2} \||\nabla|^{1/2} u\|_{L^2} \right) \left(\int_\tau^\infty m^{s-2} dm \right) \\
		&\lesssim \tau^{s-1} \|\nabla \varphi\|_{W^{1,\infty}} \left(\||\nabla|^{1/2} u\|^2_{L^2} + \|u\|_{L^2} \||\nabla|^{1/2} u\|_{L^2} \right).
		\end{align*}
		Collecting two terms, we get
		\begin{align*}
		\left| \int_0^\infty m^s \int \overline{u}_m \nabla \varphi \cdot \nabla u_m dx dm \right| \lesssim \left(\tau^{s-1/2} \|\nabla \varphi\|_{L^\infty} + \tau^{s-1} \|\nabla \varphi\|_{W^{1,\infty}} \right) \|u\|^2_{H^{1/2}},
		\end{align*}
		for any $\tau>0$. Taking $\tau=1$, we prove $(\ref{various estimate 3})$. 
	\end{proof}
	\begin{lem}[\cite{BouHimLen}] \label{lemma various estimate 4}
		Let $d\geq 1$, $s\in (0,1)$ and $\varphi: \R^d \rightarrow \R$ be such that $\Delta \varphi \in W^{2,\infty}$. Then for any $u \in L^2$, it holds that
		\begin{align}
		\left| \int_0^\infty m^s \int \Delta^2 \varphi |u_m|^2 dx dm\right| \leq C \|\Delta^2 \varphi\|^s_{L^\infty} \|\Delta \varphi\|^{1-s}_{L^\infty} \|u\|^2_{L^2}, \label{various estimate 4}
		\end{align}
		for some constant $C>0$ depending only on $s$ and $d$, where $u_m$ is given in $(\ref{define auxiliary function})$.
	\end{lem}
	We refer the reader to \cite[Appendix A]{BouHimLen} for the proofs of Lemmas $\ref{lemma various estimate 1}$ and $\ref{lemma various estimate 4}$. We also note that by using the fact 
	\[
	\frac{\sin \pi s}{\pi} \int_0^\infty \frac{m^s}{(|\xi|^2+m)^2} dm = s |\xi|^{2s-2},
	\]
	the Plancherel's and Fubini's theorems imply the following useful identity
	\begin{align}
	\begin{aligned}
	\int_0^\infty m^s \int |\nabla u_m|^2 dx dm &= \int \left(\frac{\sin \pi s}{\pi} \int_0^\infty \frac{m^s dm}{(|\xi|^2+m)^2} \right) |\xi|^2 |\hat{u}(\xi)|^2 d\xi \\
	&= \int (s|\xi|^{2s-2})|\xi|^2 |\hat{u}(\xi)|^2 d\xi=s \|(-\Delta)^{s/2} u\|^2_{L^2},
	\end{aligned}
	\label{auxiliary identity}
	\end{align}
	for any $u \in \dot{H}^s$. 
	
	Now, let $d\geq 1$, $1/2 < s <1$ and $\varphi : \R^d \rightarrow \R$ be such that $\varphi \in W^{2,\infty}$. Assume $u \in C([0,T), H^s)$ is a solution to $(\ref{focusing FNLS})$. Note that in \cite{BouHimLen}, the authors derive virial estimates by assuming that the solution $u(t)$ belongs to $H^{2s}$ for any $t\in [0,T)$. This regularity assumption is neccessary due to the lack of local theory at the time. By the local theory given in Section $\ref{section local theory}$, one can extend virial estimates to $u\in C([0,T), H^s)$ by an approximation argument. The type-I localized virial action of $u$ associated to $\varphi$ is defined by
	\begin{align}
	V_\varphi(u(t)) := \int \varphi (x) |u(t,x)|^2 dx. \label{type-I localized virial action}
	\end{align}
	\begin{lem} [Localized virial identity I] \label{lemma localized virial identity I}
		Let $d\geq 1, 1/2 < s <1$ and $\varphi: \R^d \rightarrow \R$ be such that $\varphi \in W^{2,\infty}$. Assume that $u \in C([0,T), H^s)$ is a solution to $(\ref{focusing FNLS})$. Then for any $t\in [0,T)$, it holds that
		\begin{multline}
		\frac{d}{dt} V_\varphi(u(t)) \\
		= -i\int_0^\infty m^s \int \Delta \varphi |u_m(t)|^2 dx dm -2i \int_0^\infty m^s \int \overline{u}_m(t) \nabla \varphi \cdot \nabla u_m(t) dx dm, \label{localized virial identity I}
		\end{multline}
		where $u_m(t) = c_s(-\Delta+m)^{-1} u(t)$.
	\end{lem}
	\begin{proof}
		We only verify $(\ref{localized virial identity I})$ for $u \in C^\infty_0(\R^d)$. The general case follows by an approximation argument (see \cite[Lemma 2.1]{BouHimLen}). By definition, we see that
		\[
		V_\varphi(u(t)) = \scal{u(t), \varphi u(t)},
		\]
		where $\scal{u,v}$ is the scalar product in $L^2$. Taking the time derivative and using that $u(t)$ solves $(\ref{focusing FNLS})$, we have
		\begin{align}
		\frac{d}{dt} V_\varphi(u(t)) = i\scal{u(t), [(-\Delta)^s, \varphi] u(t)}, \label{time derivative virial action}
		\end{align}
		where $[X,Y] = XY-YX$ denotes the commutator of $X$ and $Y$. To study $[(-\Delta)^s, \varphi]$, we recall the following Balakrishnan's formula 
		\begin{align}
		(-\Delta)^s =\frac{\sin \pi s}{\pi} \int_0^\infty m^{s-1} \frac{-\Delta}{-\Delta +m} dm. \label{balakrishnan formula}
		\end{align}
		This formula follows from spectral calculus applied to the self-adjoint operator $-\Delta$ and the identity
		\[
		x^s = \frac{\sin \pi s}{\pi} \int_0^\infty m^{s-1} \frac{x}{x+m} dm 
		\]
		which is avalable for any $x>0$ and $s \in (0,1)$. The Balakrishnan's representation formula $(\ref{balakrishnan formula})$  for the fractional Laplacian $(-\Delta)^s$ was firstly used in \cite{KriegerLenzmannRaphael} to study the nonlinear half-wave equation, i.e. $(\ref{focusing FNLS})$ with $s=1/2$. We also have the following commutator identity 
		\begin{align}
		\left[ \frac{A}{A+m}, B\right] = \left[\mathds{1} - \frac{m}{A+m}, B \right] = -m \left[\frac{1}{A+m}, B\right] = m \frac{1}{A+m} [A,B] \frac{1}{A+m}, \label{commutator identity}
		\end{align}
		for operators $A\geq 0$ and $B$, where $m>0$ is any positive real number. Using $(\ref{balakrishnan formula})$, we apply $(\ref{commutator identity})$ with $A=-\Delta$ to get
		\begin{align}
		[(-\Delta)^s, B] = \frac{\sin \pi s}{\pi} \int_0^\infty m^s \frac{1}{-\Delta +m} [-\Delta, B] \frac{1}{-\Delta+m} dm. \label{commutator identity fractional laplacian}
		\end{align}
		Applying the above identity with $B=\varphi$ and using the fact
		\[
		[-\Delta,\varphi] = -\Delta \varphi - 2 \nabla \varphi \cdot \nabla,
		\]
		the integration by parts yields $(\ref{localized virial identity I})$. Indeed,
		\begin{align*}
		\scal{u(t), [(-\Delta)^s, \varphi] u(t)} &= \scal{ u(t), \left( \frac{\sin \pi s}{\pi} \int_0^\infty m^s \frac{1}{-\Delta +m} [-\Delta, \varphi] \frac{1}{-\Delta+m} dm\right) u(t) } \\
		&= \frac{\sin \pi s}{\pi} \int_0^\infty m^s \scal{u(t), \frac{1}{-\Delta +m} [-\Delta, \varphi] \frac{1}{-\Delta +m} u(t) } dm \\
		&= \int_0^\infty m^s \scal{c_s (-\Delta+m)^{-1} u(t), [-\Delta, \varphi] c_s (-\Delta+m)^{-1} u(t)} dm \\
		&= \int_0^\infty m^s \int \overline{u}_m(t) (-\Delta \varphi u_m(t) -2 \nabla \varphi \cdot \nabla u_m(t)) dx dm\\
		&= \int_0^\infty m^s \int \left( -\Delta \varphi |u_m(t)|^2  - 2 \overline{u}_m(t) \nabla \varphi \cdot \nabla u_m(t) \right) dx dm.
		\end{align*}
		The proof is complete.		
	\end{proof}
	A direct consequence of Lemmas $\ref{lemma various estimate 2}$, $\ref{lemma various estimate 3}$, $\ref{lemma localized virial identity I}$ and the fact $\|\nabla \varphi\|_{W^{1,\infty}} \sim \|\nabla \varphi\|_{L^\infty} + \|\Delta \varphi\|_{L^\infty}$ is the following estimate.
	\begin{coro} \label{corollary localized virial estimate 1}
		Let $d\geq 1$, $1/2<s<1$ and $\varphi:\R^d \rightarrow \R$ be such that $\varphi \in W^{2,\infty}$. Assume that $u \in C([0,T), H^s)$ is a solution to $(\ref{focusing FNLS})$. Then for any $t\in [0,T)$,
		\begin{align}
		\left| \frac{d}{dt} V_\varphi (u(t)) \right| \leq  C\|\nabla \varphi\|_{W^{1,\infty}} \|u(t)\|^2_{H^s},
		\end{align}
		for some constant $C>0$ depending only on $s$ and $d$. 
	\end{coro}
	
	We next define the type-II localized virial action of $u$ associated to $\varphi$ by
	\begin{align}
	M_\varphi(u(t)):= 2 \im{ \int \overline{u}(t,x) \nabla \varphi (x) \cdot \nabla u(t,x)  dx}. \label{localized virial action}
	\end{align}
	Thanks to Lemma $\ref{lemma various estimate 1}$, the quantity $M_\varphi(u(t))$ is well-defined. Indeed, by $(\ref{various estimate 1})$,
	\[
	|M_\varphi(u(t))| \lesssim C(\|\nabla \varphi\|_{L^\infty}, \|\Delta \varphi\|_{L^\infty}) \|u(t)\|^2_{H^{1/2}} \lesssim C(\varphi) \|u(t)\|^2_{H^s} <\infty.
	\]
	We have the following virial identity (see \cite[Lemma 2.1]{BouHimLen}).
	\begin{lem} [Localized virial identity II \cite{BouHimLen}] \label{lemma localized virial identity}
		Let $d\geq 1, 1/2 < s <1$ and $\varphi: \R^d \rightarrow \R$ be such that $\nabla \varphi \in W^{3,\infty}$. Assume that $u \in C([0,T), H^s)$ is a solution to $(\ref{focusing FNLS})$. Then for any $t\in [0,T)$, it holds that
		\begin{align}
		\frac{d}{dt} M_\varphi(u(t)) &= - \int_0^\infty m^s \int \Delta^2 \varphi |u_m(t)|^2 dx dm \nonumber\\
		&\mathrel{\phantom{=}} + 4 \sum_{j,k=1}^d \int_0^\infty m^s \int \partial^2_{jk} \varphi \partial_j \overline{u}_m(t) \partial_k u_m(t) dx dm \label{localized virial identity II} \\
		&\mathrel{\phantom{=}} - \frac{2\alpha}{\alpha+2} \int \Delta \varphi |u(t)|^{\alpha+2} dx, \nonumber
		\end{align}
		where $u_m(t) = c_s(-\Delta+m)^{-1} u(t)$.
	\end{lem}
\begin{rem} \label{remark localized virial identity}
	If we make the formal substitution and take the unbounded function $\varphi(x) = |x|^2$, then by $(\ref{auxiliary identity})$, we obtain the virial identity
	\begin{align}
	\frac{d}{dt} \left(4 \im \int \overline{u}(t) x \cdot \nabla u(t) dx\right) &= 8s \|(-\Delta)^{s/2} u(t)\|^2_{L^2} - \frac{4d\alpha}{\alpha+2} \|u(t)\|^{\alpha+2}_{L^{\alpha+2}} \label{virial identity}  \\
	&= 4d\alpha E(u(t)) -2(d\alpha-4s)\|(-\Delta)^{s/2} u(t)\|^2_{L^2}. 
	\nonumber
	\end{align}
	This identity can be proved rigorously by integrating $(\ref{focusing FNLS})$ against $i\left( x\cdot \nabla + \nabla \cdot x \right) \overline{u}(t)$ on $\R^d$. 
\end{rem}
\section{Blow-up criteria} \label{section blowup criteria}
\setcounter{equation}{0}
In this section, we give the proof of Theorem $\ref{theorem blowup criteria}$ and its applications. We follow closely the argument of \cite{DuWuZhang}. 

\subsection{Proof of Theorem $\ref{theorem blowup criteria}$}  
If $T<+\infty$, then we are done. If $T=+\infty$, we show $(\ref{divergence Lq norm})$. By contradiction, we assume that the solution exists globally in time and there exists $q>\alpha+2$ such that
\begin{align}
\sup_{t\in [0,+\infty)} \|u(t)\|_{L^q} <\infty. \label{bounded Lq norm}
\end{align}
Interpolating between $L^2$ and $L^q$, the conservation of mass implies
\begin{align*}
\sup_{t\in [0,+\infty)} \|u(t)\|_{L^{\alpha+2}} <\infty. 
\end{align*}
By the conservation of mass and energy, we get
\begin{align}
\sup_{t\in [0,+\infty)} \|u(t)\|_{H^s} <\infty. \label{bounded Hs norm}
\end{align}
As in \cite{DuWuZhang}, the first step is to control $L^2$-norm of the solution outside a large ball. To do so, we introduce $\vartheta: [0,\infty) \rightarrow [0,1]$ a smooth function satisfying
\[
\vartheta(r)= \left\{
\begin{array}{cl}
0 &\text{if } 0 \leq r \leq \frac{1}{2}, \\
1 &\text{if } r \geq 1.
\end{array}
\right.
\]
Given $R>1$, we denote the radial function 
\[
\psi_R(x) = \psi_R(r) := \vartheta(r/R), \quad r=|x|. 
\]
It is easy to check that
\[
\nabla \psi_R(x) = \frac{x}{rR} \vartheta'(r/R), \quad \Delta \psi_R(x) = \frac{1}{R^2} \vartheta''(r/R) + \frac{(d-1)}{rR} \vartheta'(r/R).
\]
We thus get
\begin{align}
\|\nabla \psi_R\|_{W^{1,\infty}} \sim \|\nabla \psi_R\|_{L^\infty} + \|\Delta \psi_R\|_{L^\infty} \lesssim R^{-1} + R^{-2} \lesssim R^{-1}. \label{estimate psi_R}
\end{align}
We next define 
\[
V_{\psi_R}(u(t)) := \int \psi_R(x) |u(t,x)|^2 dx.
\]
By the fundamental theorem of calculus, we have
\[
V_{\psi_R}(u(t)) = V_{\psi_R}(u_0) + \int_0^t \frac{d}{d\tau} V_{\psi_R}(u(\tau)) d\tau \leq V_{\psi_R}(u_0) + \left(\sup_{\tau \in [0,t]} \left|\frac{d}{d\tau} V_{\psi_R}(u(\tau)) \right| \right) t. 
\]
Using Corollary $\ref{corollary localized virial estimate 1}$,$(\ref{estimate psi_R})$ and $(\ref{bounded Hs norm})$, we get
\begin{align*}
\sup_{\tau \in [0,t]} \left|\frac{d}{d\tau} V_{\psi_R}(u(\tau)) \right|  \lesssim  \|\nabla \psi_R\|_{W^{1,\infty}} \sup_{\tau \in [0,t]} \|u(\tau)\|^2_{H^s} \leq CR^{-1},
\end{align*}
for some constant $C$ independent of $R$. We thus obtain
\[
V_{\psi_R}(u(t)) \leq V_{\psi_R}(u_0) + CR^{-1} t.
\]
By the choice of $\vartheta$, the conservation of mass yields
\[
V_{\psi_R}(u_0) = \int \psi_R(x) |u_0(x)|^2 dx \leq \int_{|x|>R/2} |u_0(x)|^2 dx \rightarrow 0,
\]
as $R\rightarrow \infty$ or $V_{\psi}(u_0) = o_R(1)$. Using the fact
\[
\int_{|x|\geq R} |u(t,x)|^2 dx \leq V_{\psi_R}(u(t)),
\]
we obtain the following control on the $L^2$-norm of $u$ outside a large ball. 
\begin{lem}[$L^2$-norm outside a large ball] \label{lemma L2 norm outside large ball}
	Let $\vareps>0$ and $R>1$. Then there exists a constant $C>0$ independent of $R$ such that for any $t \in [0,T_0]$ with $T_0:= \frac{\vareps R}{C}$, 
	\begin{align}
	\int_{|x|\geq R} |u(t,x)|^2 dx \leq o_R(1) + \vareps. \label{L2 norm outside large ball}
	\end{align}
\end{lem}
Next, let us choose $\theta: [0,\infty) \rightarrow [0,\infty)$ a smooth function such that
\[
\theta(r) = \left\{
\begin{array}{cl}
r^2 &\text{if } 0 \leq r \leq 1, \\
2 &\text{if } r\geq 2,
\end{array}
\right.
\quad \text{and} \quad \theta''(r) \leq 2 \text{ for } r\geq 0.
\]
Given $R>1$, we define the radial function
\begin{align}
\varphi_R(x) = \varphi_R(r):= R^2\theta(r/R), \quad r=|x|. \label{define varphi_R}
\end{align}
We readily verify that
\begin{align}
2-\varphi''_R(r) \geq 0, \quad 2-\frac{\varphi'_R(r)}{r} \geq 0, \quad 2d-\Delta \varphi_R(x) \geq 0, \quad \forall r\geq 0, \quad \forall x \in \R^d. \label{properties varphi_R}
\end{align}
Moreover,
\[
\|\nabla^k \varphi_R\|_{L^\infty} \lesssim R^{2-k}, \quad k = 0, \cdots, 4,
\]
and
\[
\text{supp}(\nabla^k \varphi_R) \subset \left\{
\begin{array}{cl}
\{|x| \leq 2R\} &\text{for } k=1,2, \\
\{R \leq |x| \leq 2R\} &\text{for } k=3,4.
\end{array}
\right.
\]
Denote
\[
M_{\varphi_R}(u(t)) := 2 \im{ \int \overline{u}(t,x) \nabla \varphi_R(x) \cdot \nabla u(t,x) dx}.
\]
Applying Lemma $\ref{lemma localized virial identity}$ with $\varphi(x) = \varphi_R(x)$, we have
\begin{align*}
\frac{d}{dt} M_{\varphi_R}(u(t)) &= -\int_0^\infty m^s \int \Delta^2\varphi_R |u_m(t)|^2 dx dm \\
&\mathrel{\phantom{=}} + 4 \sum_{j,k=1}^d \int_0^\infty m^s \int \partial^2_{jk} \varphi_R \partial_j \overline{u}_m(t) \partial_k u_m(t) dx dm \\
&\mathrel{\phantom{=}} -\frac{2\alpha}{\alpha+2} \int \Delta \varphi_R |u(t)|^{\alpha+2} dx,
\end{align*}
where $u_m(t) = c_s (-\Delta+m)^{-1} u(t)$. Since $\text{supp}(\Delta^2 \varphi_R) \subset \{ |x| \geq R\}$, we use Lemma $\ref{lemma various estimate 4}$ to have
\begin{align*}
\left|\int_0^\infty m^s \int \Delta^2 \varphi_R |u_m(t)|^2 dx dm \right| & \lesssim \|\Delta^2 \varphi_R\|^s_{L^\infty} \|\Delta \varphi_R\|^{1-s}_{L^\infty} \|u(t)\|^2_{L^2(|x|\geq R)} \\
&\lesssim R^{-2s} \|u(t)\|^2_{L^2(|x| \geq R)}.
\end{align*}
Since $\varphi_R$ is radial, we use the fact
\[
\partial^2_{jk} = \left(\frac{\delta_{jk}}{r} - \frac{x_jx_k}{r^3} \right) \partial_r + \frac{x_j x_k}{r^2} \partial^2_r
\]
to write
\begin{multline*}
\sum_{j,k=1}^d \int_0^\infty m^s \int \partial^2_{jk} \varphi_R \partial_j \overline{u}_m (t) \partial_k u_m(t) dx dm = \int_0^\infty m^s \int \frac{\varphi'_R}{r} |\nabla u_m(t)|^2 dx dm \\
 + \int_0^\infty m^s \int \left(\frac{\varphi''_R}{r^2} - \frac{\varphi'_R}{r^3} \right) |x\cdot \nabla u_m(t)|^2 dx dm.
\end{multline*}
Thanks to the identity $(\ref{auxiliary identity})$, we have
\begin{multline*}
\int_0^\infty m^s \int \frac{\varphi'_R}{r} |\nabla u_m(t)|^2 dx dm = 2s \|(-\Delta)^{s/2} u(t)\|^2_{L^2} \\
+  \int_0^\infty m^s \int \left(\frac{\varphi'_R}{r}-2 \right) |\nabla u_m(t)|^2 dx dm.
\end{multline*}
We next use the fact $\varphi''_R \leq 2$ and the Cauchy-Schwarz estimate $|x\cdot \nabla u_m| \leq r|\nabla u_m|$ to see that
\begin{multline*}
\int_0^\infty m^s \int \left( \frac{\varphi'_R}{r} -2\right) |\nabla u_m(t)|^2 dx dm \\
+ \int_0^\infty m^s \int \left(\varphi''_R - \frac{\varphi'_R}{r} \right) \frac{|x\cdot \nabla u_m(t)|^2}{r^2} dx dm \leq 0.
\end{multline*}
We next write
\[
-\frac{2\alpha}{\alpha+2} \int \Delta \varphi_R |u(t)|^{\alpha+2} dx = -\frac{4d\alpha}{\alpha+2} \|u(t)\|^{\alpha+2}_{L^{\alpha+2}} + \frac{2\alpha}{\alpha+2} \int (2d-\Delta \varphi_R) |u(t)|^{\alpha+2} dx.
\]
Collecting the above estimates, we obtain
\begin{align*}
\frac{d}{dt} M_{\varphi_R} (u(t)) &\leq 8s \|(-\Delta)^{s/2} u(t)\|^2_{L^2} -\frac{4d\alpha}{\alpha+2} \|u(t)\|^{\alpha+2}_{L^{\alpha+2}} + C R^{-2s} \|u(t)\|^2_{L^2(|x| \geq R)} \\
&\mathrel{\phantom{\leq 8s \|(-\Delta)^{s/2} u(t)\|^2_{L^2} }} + \frac{2\alpha}{\alpha+2} \int (2d - \Delta \varphi_R) |u(t)|^{\alpha+2} dx.
\end{align*}
Since $\text{supp}(2d-\Delta\varphi_R) \subset \{|x| \geq R\}$ and $\|2d-\Delta\varphi_R\|_{L^\infty} \lesssim 1$, we interpolate between $L^2$ and $L^q$ and use $(\ref{bounded Lq norm})$ to get
\[
\int (2d-\Delta \varphi_R) |u(t)|^{\alpha+2} dx \lesssim \|u(t)\|^{(1-\eta)(\alpha+2)}_{L^q(|x| \geq R)} \|u(t)\|^{\eta(\alpha+2)}_{L^2(|x| \geq R)} \lesssim \|u(t)\|^{\eta(\alpha+2)}_{L^2(|x|\geq R)},
\]
for some $0<\eta<1$. Note that the condition $q>\alpha+2$ is neccessary in the above estimate. We thus obtain the following estimate.
\begin{lem} \label{lemma localized virial estimate application}
	Let $R>1$ and $\varphi_R$ be as in $(\ref{define varphi_R})$. There exist a constant $C>0$ independent of $R$ and $0<\eta<1$ such that
	\begin{align}
	\frac{d}{dt} M_{\varphi_R}(u(t)) \leq 16 K(u(t)) + CR^{-2} \|u(t)\|^2_{L^2(|x| \geq R)} + C \|u(t)\|^{\eta(\alpha+2)}_{L^2(|x|\geq R)}, \label{localized virial estimate application}
	\end{align}
	for any $t\in [0,T)$, where $K(u(t))$ is given in $(\ref{define blowup quantity})$.
\end{lem}
We now complete the proof of Theorem $\ref{theorem blowup criteria}$. Applying Lemma $\ref{lemma L2 norm outside large ball}$ and Lemma $\ref{lemma localized virial estimate application}$, we see that for any $\vareps>0$ and any $R>1$, there exists a constant $C>0$ independent of $R$ such that for any $t\in [0,T_0]$ with $T_0:=\frac{\vareps R}{C}$,
\begin{align*}
\frac{d}{dt} M_{\varphi_R}(u(t)) &\leq 16 K(u(t)) + CR^{-2}(o_R(1)+\vareps)^2+ C(o_R(1) +\vareps)^{\eta(\alpha+2)} \\
&\leq -16\delta + CR^{-2}\left(o_R(1) + \vareps^2\right) + C\left(o_R(1) + \vareps^{\eta(\alpha+2)}\right).
\end{align*}
Note that the constant $C$ may change from lines to lines but is independent of $R$. We now choose $\vareps>0$ so that
\begin{align*}
C \vareps^{\eta(\alpha+2)} = 4\delta. 
\end{align*}
We see that for $R\gg 1$ large,
\begin{align}
\frac{d}{dt} M_{\varphi_R}(u(t)) \leq -\delta <0, \label{negativity derivative virial action}
\end{align}
for any $t\in [0,T_0]$ with $T_0 = \frac{\vareps R}{C}$. Note also that since $\vareps>0$ is fixed, we can take $T_0$ as large as we want by increasing $R$ accordingly. From $(\ref{negativity derivative virial action})$, we infer that 
\begin{align}
M_{\varphi_R}(u(t)) \leq -c t, \label{estimate virial action}
\end{align}
for all $t\in [t_0,T_0]$ with some sufficiently large time $t_0 \in [0,T_0]$ and some constant $c>0$ depending only on $\delta$. On the other hand, by Lemma $\ref{lemma various estimate 1}$ and the conservation of mass, we see that for \footnote{The solution is assumed to exist on $[0,+\infty)$.} any $t\in [0,+\infty)$,
\begin{align*}
|M_{\varphi_R}(u(t))| &\lesssim C(\varphi_R) \left(\||\nabla|^{1/2} u(t)\|^2_{L^2} + \|u(t)\|_{L^2} \||\nabla|^{1/2} u(t)\|_{L^2} \right) \\
&\lesssim C(\varphi_R) \left(\||\nabla|^{1/2}u(t)\|^2_{L^2} + \|u(t)\|^2_{L^2}\right) \\
&\lesssim C(\varphi_R) \left(\||\nabla|^{1/2} u(t)\|^2_{L^2} +1 \right) \\
&\lesssim C(\varphi_R) \left(\|(-\Delta)^{s/2}u(t)\|^{\frac{1}{s}}_{L^2} +1 \right).
\end{align*}
Here we use the interpolation estimate $\||\nabla|^{1/2} u\|_{L^2} \lesssim \|(-\Delta)^{s/2} u\|^{\frac{1}{2s}}_{L^2} \|u\|^{1-\frac{1}{2s}}_{L^2}$ for $s>1/2$. This combined with $(\ref{estimate virial action})$ yield
\begin{align*}
c t \leq - M_{\varphi_R}(u(t)) = |M_{\varphi_R}(u(t))| \lesssim C(\varphi_R) \left( \|(-\Delta)^{s/2} u(t)\|^{\frac{1}{s}}_{L^2}+1 \right),
\end{align*}
for any $t\in [t_0, T_0]$. This shows that
\begin{align}
\|(-\Delta)^{s/2} u(t)\|_{L^2} \geq C t^s, \label{lower bound Hs norm}
\end{align}
for any $t\in [\overline{t}_0, T_0]$ with some sufficiently large time $t_0 \leq \overline{t}_0 \leq T_0$. Taking $t \uparrow T_0= \frac{\vareps R}{C}$, we see that 
\[
\|u(t)\|_{H^s} \rightarrow \infty \text{ as } R \rightarrow \infty,
\]
which contradicts to $(\ref{bounded Hs norm})$. The proof is complete. 
\defendproof

\subsection{Mass-critical blow-up criteria}
In this short subsection, we give the proof of Proposition $\ref{proposition mass-critical blowup criteria}$. This result follows directly from Corollary $\ref{corollary negative energy blowup criteria}$ with $\gamma=s$. Moreover, if $T<+\infty$, the limit 
\[
\lim_{t\uparrow T} \|(-\Delta)^{s/2} u(t)\|_{L^2} = \infty
\]
follows from the blow-up alternative (see Section $\ref{section local theory}$). In the case $T=+\infty$, the Sobolev embedding, namely $H^s(\R^d) \subset L^q(\R^d)$ for any $q \in [2, \infty)$ satisfying $\frac{1}{q} \geq \frac{1}{2}-\frac{s}{d}$ together with the conservation of mass show
\[
\sup_{t\in [0,+\infty)} \|(-\Delta)^{s/2} u(t)\|_{L^2} =\infty.
\]
The conservation of energy then yields
\[
\sup_{t\in [0,+\infty)} \|u(t)\|_{L^{\frac{4s}{d}+2}} =\infty.
\]
This proves Proposition $\ref{proposition mass-critical blowup criteria}$.
\defendproof

\subsection{Mass and energy intercritical blow-up criteria}
We now give the proof of Proposition $\ref{proposition intercritical blowup criteria}$. By the same argument as in the previous subsection using the Sobolev embedding and the conservation of mass and energy, it remains to show $(\ref{blowup condition})$ for some $\delta>0$. The case $E(u_0)<0$ follows as in Corollary $\ref{corollary negative energy blowup criteria}$. Let us now consider initial data $u_0$ with $E(u_0) \geq 0$ and  $(\ref{blowup condition intercritical})$. The assumption $(\ref{blowup condition intercritical})$ implies
\begin{align}
\left\{
\begin{array}{ccc}
E(u_0) M^\sigma(u_0) &< & E(Q) M^\sigma(Q), \\
\|(-\Delta)^{s/2} u_0\|_{L^2} \|u_0\|^\sigma_{L^2} &>& \|(-\Delta)^{s/2} Q\|_{L^2} \|Q\|^\sigma_{L^2},
\end{array}
\right.
\label{blowup condition intercritical equivalence}
\end{align}
where 
\[
\sigma:= \frac{s-\sct}{\sct} = \frac{4s-(d-2s)\alpha}{d\alpha-4s}. 
\]
We next recall the sharp Gagliardo-Nirenberg inequality (see e.g. \cite[Appendix]{BouHimLen})
\begin{align}
\|u(t)\|^{\alpha+2}_{L^{\alpha+2}} \leq C_{\text{GN}} \|u(t)\|^{\frac{4s-(d-2s)\alpha}{2s}}_{L^2} \|(-\Delta)^{s/2} u(t)\|^{\frac{d\alpha}{2s}}_{L^2}, \label{sharp gagliardo-nirenberg inequality}
\end{align}
where the sharp constant is given by
\begin{align}
C_{\text{GN}} = \frac{\|Q\|^{\alpha+2}_{L^{\alpha+2}}}{ \|Q\|^{\frac{4s-(d-2s)\alpha}{2s}}_{L^2} \|(-\Delta)^{s/2} Q\|^{\frac{d\alpha}{2s}}_{L^2}}, \label{sharp constant gagliardo-nirenberg inequality}
\end{align}
with $Q$ is the unique (up to symmetries) positive radial solution to $(\ref{elliptic equation intercritical})$. We also have the following Pohozaev's identities
\begin{align}
\|(-\Delta)^{s/2} Q\|^2_{L^2} = \frac{d\alpha}{2s(\alpha+2)} \|Q\|^{\alpha+2}_{L^{\alpha+2}} = \frac{d\alpha}{4s-(d-2s)\alpha} \|Q\|^2_{L^2}. \label{pohozaev identities intercritical}
\end{align}
A direct calculation shows
\begin{align}
C_{\text{GN}} &= \frac{2s(\alpha+2)}{d\alpha} \frac{1}{\left( \|(-\Delta)^{s/2} Q\|_{L^2} \|Q\|^\sigma_{L^2} \right)^{\frac{d\alpha-4s}{2s}} } \label{sharp constant relation} \\
E(Q) M^\sigma(Q) &= \frac{d\alpha-4s}{2d\alpha} \left( \|(-\Delta)^{s/2} Q\|_{L^2} \|Q\|^\sigma_{L^2} \right)^2. \label{threshold relation intercritical}
\end{align}
We now multiply both sides of $E(u(t))$ by $M^\sigma(u(t))$ and use the sharp Gagliardo-Nirenberg inequality to get
\begin{align*}
E(u(t)) M^\sigma(u(t)) &= \frac{1}{2} \left( \|(-\Delta)^{s/2} u(t)\|_{L^2} \|u(t)\|^\sigma_{L^2} \right)^2 - \frac{1}{\alpha+2} \|u(t)\|^{\alpha+2}_{L^{\alpha+2}} \|u(t)\|^{2\sigma}_{L^2} \\
&\geq \frac{1}{2} \left( \|(-\Delta)^{s/2} u(t)\|_{L^2} \|u(t)\|^\sigma_{L^2}\right)^2 \\
&\mathrel{\phantom{\geq}}- \frac{C_{\text{GN}}}{\alpha+2} \|(-\Delta)^{s/2} u(t)\|^{\frac{d\alpha}{2s}}_{L^2} \|u(t)\|^{\frac{4s-(d-2s)\alpha}{2s}}_{L^2} \\
&= f\left(  \|(-\Delta)^{s/2} u(t)\|_{L^2} \|u(t)\|^\sigma_{L^2} \right),
\end{align*}
where $f(x) := \frac{1}{2} x^2 - \frac{C_{\text{GN}}}{\alpha+2} x^{\frac{d\alpha}{2s}}$. It is easy to see that $f$ is increasing on $(0,x_0)$ and decreasing on $(x_0, \infty)$, where
\[
x_0 = \left( \frac{2s(\alpha+2)}{d\alpha C_{\text{GN}}} \right)^{\frac{2s}{d\alpha-4s}}= \|(-\Delta)^{s/2} Q\|_{L^2} \|Q\|^\sigma_{L^2}.
\]
Here the last equality follows from $(\ref{sharp constant relation})$. By $(\ref{sharp constant relation})$ and $(\ref{threshold relation intercritical})$, we see that
\begin{align}
f\left(\|(-\Delta)^{s/2} Q\|_{L^2} \|Q\|^\sigma_{L^2} \right) &= \frac{d\alpha-4s}{2d\alpha} \left( \|(-\Delta)^{s/2} Q\|_{L^2} \|Q\|^\sigma_{L^2} \right)^2  \nonumber \\
&= E(Q) M^\sigma(Q). \label{threshold relation intercritical 1}
\end{align}
Thus the conservation of mass and energy toghether with the first condition in $(\ref{blowup condition intercritical equivalence})$ imply
\begin{align*}
f\left(  \|(-\Delta)^{s/2} u(t)\|_{L^2} \|u(t)\|^\sigma_{L^2} \right) &\leq E(u(t)) M^\sigma(u(t)) = E(u_0) M^\sigma(u_0) \\
&< E(Q) M^\sigma (Q) = f\left(\|(-\Delta)^{s/2} Q\|_{L^2} \|Q\|^\sigma_{L^2} \right).
\end{align*}
Using the second condition $(\ref{blowup condition intercritical equivalence})$, the continuity argument shows that
\begin{align}
\|(-\Delta)^{s/2} u(t)\|_{L^2} \|u(t)\|^\sigma_{L^2} > \|(-\Delta)^{s/2} Q\|_{L^2} \|Q\|^\sigma_{L^2}, \label{property solution intercritical}
\end{align}
for any $t\in [0,T)$. This implies that there exists $\delta>0$ so that $(\ref{blowup condition})$ holds. Indeed, since $E(u_0) M^\sigma(u_0) < E(Q) M^\sigma(Q)$, we pick $\rho>0$ small enough so that 
\begin{align}
E(u_0) M^\sigma(u_0) \leq (1-\rho) E(Q) M^\sigma(Q). \label{refined blowup condition intercritical}
\end{align}
Multiplying $K(u(t))$ with the conserved quantity $M^\sigma(u(t))$ and using $(\ref{threshold relation intercritical 1})$, $(\ref{property solution intercritical})$ and $(\ref{refined blowup condition intercritical})$, we obtain
\begin{align*}
K(u(t)) M^\sigma(u(t)) &= \frac{d\alpha}{4} E(u(t)) M^\sigma(u(t)) -\frac{d\alpha-4s}{8} \left(\|(-\Delta)^{s/2} u(t)\|_{L^2} \|u(t)\|^\sigma_{L^2} \right)^2 \\
&= \frac{d\alpha}{4} E(u_0) M^\sigma(u_0) - \frac{d\alpha-4s}{8} \left(\|(-\Delta)^{s/2} u(t)\|_{L^2} \|u(t)\|^\sigma_{L^2} \right)^2 \\ 
&\leq \frac{d\alpha}{4} (1-\rho) E(Q) M^\sigma(Q) - \frac{d\alpha-4s}{8} \left(\|(-\Delta)^{s/2} Q\|_{L^2} \|Q\|^\sigma_{L^2} \right)^2 \\
&= -\frac{(d\alpha-4s)\rho}{8} \left(\|(-\Delta)^{s/2} Q\|_{L^2} \|Q\|^\sigma_{L^2} \right)^2,
\end{align*}
for any $t\in [0,T)$. This shows $(\ref{blowup condition})$ with 
\[
\delta = \frac{(d\alpha-4s)\rho}{8} \|(-\Delta)^{s/2} Q\|^2_{L^2} \left(\frac{M(Q)}{M(u_0)} \right)^\sigma>0.
\]
The proof is complete.
\defendproof

\subsection{Energy critical blow-up criteria}
In this subsection, we give the proof of Proposition $\ref{proposition energy-critical blowup criteria}$. The proof is similar to the one of Proposition $\ref{proposition intercritical blowup criteria}$. Instead of using the sharp Gagliardo-Nirenberg inequality, we make use of the sharp Sobolev embedding
\[
\|u\|_{L^{s^\star}} \leq C_{\text{SE}} \|(-\Delta)^{s/2} u\|_{L^2},
\]
where $s^\star = \frac{2d}{d-2s}$ and the sharp constant
\begin{align}
C_{\text{SE}} = \frac{\|W\|_{L^{s^\star}}}{\|(-\Delta)^{s/2} W\|_{L^2}}. \label{sharp constant sobolev embedding}
\end{align}
Here $W$ is the unique (up to symmetries) positive radial solution to $(\ref{elliptic equation energy-critical})$. The following identities are easy to check
\begin{align}
\|(-\Delta)^{s/2} W\|^2_{L^2} &= \|W\|^{s^\star}_{L^{s^\star}} = \frac{1}{C_{\text{SE}}^{\frac{d}{s}}}, \label{pohozaev identities energy-critical} \\
E(W) = \frac{1}{2} \|(-\Delta)^{s/2} W\|^2_{L^2} &- \frac{1}{s^\star} \|W\|^{s^\star}_{L^{s^\star}} = \frac{s}{d} \frac{1}{C_{\text{SE}}^{\frac{d}{s}}}. \label{threshold relation energy-critical}
\end{align}
In particular, we have
\begin{align}
C_{\text{SE}} = \|(-\Delta)^{s/2} W\|^{-\frac{2s}{d}}_{L^2} = \|W\|^{-\frac{s^\star s}{d}}_{L^{s^\star}} = \left[ \frac{s}{d E(W)} \right]^{\frac{s}{d}}. \label{sharp constant relation sobolev embedding}
\end{align}
We now apply the sharp Sobolev embedding to get
\begin{align*}
E(u(t)) &= \frac{1}{2} \|(-\Delta)^{s/2} u(t)\|^2_{L^2} - \frac{1}{s^\star} \|u(t)\|^{s^\star}_{L^{s^\star}} \\
&\geq \frac{1}{2} \|(-\Delta)^{s/2} u(t)\|^2_{L^2} - \frac{[C_{\text{SE}}]^{s^\star}}{s^\star} \|(-\Delta)^{s/2} u(t)\|^{s^\star}_{L^2} = g\left(\|(-\Delta)^{s/2} u(t)\|_{L^2} \right),
\end{align*}
where $g(y) := \frac{1}{2} y^2 - \frac{[C_{\text{SE}}]^{s^\star}}{s^\star} y^{s^\star}$. We see that $g$ is increasing on $(0,y_0)$ and decreasing on $(y_0, \infty)$ with
\[
y_0 = \left( \frac{1}{[C_{\text{SE}}]^{s^\star}} \right)^{\frac{d-2s}{4s}} = \|(-\Delta)^{s/2} W\|_{L^2}.
\]
Here we use $(\ref{pohozaev identities energy-critical})$ to have the second equality. We also have from $(\ref{pohozaev identities energy-critical})$ and $(\ref{threshold relation energy-critical})$ that 
\begin{align}
g\left( \|(-\Delta)^{s/2} W\|_{L^2} \right) = \frac{s}{d} \frac{1}{C_{\text{SE}}^{\frac{d}{s}}} = E(W). \label{threshold relation energy-critical 1}
\end{align}
Thanks to the conservation of energy, the first condition in $(\ref{blowup condition energy-critical})$ yields
\begin{align*}
g\left(\|(-\Delta)^{s/2} u(t)\|_{L^2} \right) \leq E(u(t)) = E(u_0) <E(W) = g \left( \|(-\Delta)^{s/2} W\|_{L^2}\right),
\end{align*}
for any $t\in[0,T)$. By the second condition in $(\ref{blowup condition energy-critical})$, the continuity argument implies that
\begin{align}
\|(-\Delta)^{s/2} u(t)\|_{L^2} > \|(-\Delta)^{s/2} W\|_{L^2}, \label{property solution energy-critical}
\end{align}
for any $t\in [0,T)$. We next pick $\rho>0$ small enough so that
\begin{align}
E(u_0) \leq (1-\rho) E(W). \label{refined blowup condition energy-critical}
\end{align}
By the conservation of energy, $(\ref{property solution energy-critical})$, $(\ref{refined blowup condition energy-critical})$ and the fact $E(W) = \frac{s}{d}\|(-\Delta)^{s/2} W\|^2_{L^2}$, we learn that
\begin{align*}
K(u(t)) &= \frac{ds}{d-2s} E(u(t)) - \frac{s^2}{d-2s} \|(-\Delta)^{s/2} u(t)\|^2_{L^2} \\
&= \frac{ds}{d-2s} E(u_0) - \frac{s^2}{d-2s} \|(-\Delta)^{s/2} u(t)\|^2_{L^2} \\
&\leq \frac{ds}{d-2s} (1-\rho) E(W) - \frac{s^2}{d-2s} \|(-\Delta)^{s/2} W\|^2_{L^2} \\
&=-\frac{\rho s^2}{d-2s} \|(-\Delta)^{s/2} W\|^2_{L^2},
\end{align*}
for any $t\in [0,T)$. This shows $(\ref{blowup condition})$ with 
\[
\delta= \frac{\rho s^2}{d-2s} \|(-\Delta)^{s/2} W\|^2_{L^2}>0.
\]
The proof is complete.
\defendproof

\section*{Acknowledgments}
The author would like to express his deep gratitude to his wife - Uyen Cong for her encouragement and support. He would like to thank his supervisor Prof. Jean-Marc Bouclet for the kind guidance and constant encouragement. He also would like to thank the reviewer for his/her helpful comments and suggestions. 


\end{document}